\documentclass{article}       
\usepackage{amsmath,amssymb,amsthm}
\usepackage{comment}

\numberwithin{equation}{section}

\newtheorem{theorem}{Theorem}[section]
\newtheorem{proposition}{Proposition}[section]
\newtheorem{lemma}{Lemma}[section]

\DeclareMathOperator{\dist}{d}

\begin{document}

\title{New Results for the Two-Stage Contact Process}


\author{Eric Foxall}



\maketitle

\begin{abstract}
Here we continue the work started by Steve Krone on the two-stage contact process.  We give a simplified proof of the duality relation, and answer most of the open questions posed in that paper.  We also fill in the details of an incomplete proof.
\end{abstract}

\section{Introduction}
We consider the two-stage contact process introduced in \cite{krone}.  It is a natural generalization of the contact process in which there is an intermediate juvenile type that must mature before it can produce offspring.  More precisely, it is a growth model on $\mathbb{Z}^d$ defined by the rates
\begin{eqnarray*}
0 & \rightarrow & 1 \textrm{ at rate } \lambda n_2\\
1 &\rightarrow & 2 \textrm{ at rate } \gamma\\
1 &\rightarrow & 0 \textrm{ at rate } 1 + \delta\\
2 &\rightarrow & 0 \textrm{ at rate } 1
\end{eqnarray*}
where $n_2(x)$ is the cardinality of the set $\{y \in \mathbb{Z}^d:0<\|y-x\|_{\infty} < r\}$ for some $r \geq 1$.  The state space for the process is $\smash{\{0,1,2\}^{Z^d}}$, so that each site is either unoccupied, recently occupied, or occupied by a mature organism that can give birth at other sites.  Aside from the choice of neighbourhood, there are three parameters $\lambda,\gamma$ and $\delta$, respectively the transmission rate, the maturation rate, and the juvenile death rate.\\

A number of basic properties of the process are proved in \cite{krone}, including additivity and monotonicity with respect to parameters (increasing in $\lambda$ and $\gamma$ and decreasing in $\delta$), as well as a duality relation, and some bounds on the survival region (the set of parameters for which an initially finite population has a chance of surviving for all time).\\

Here we consider the process in the more general setting of a countable graph $(V,E)$ with finite maximum degree, proving some results in this setting and more precise results on $\mathbb{Z}^d$.  We simplify the proof of the duality relation given in \cite{krone} and answer most of the open questions posed in Section 4 of that paper.  As we shall see, for the two-stage contact process there is a critical value of the maturation rate below which survival does not occur (Theorem \ref{thmcritmat}).  Also, it shares many of the properties of the contact process; in particular, there is complete convergence (Theorem \ref{thmcomp}).  The following is a summary of the main results.\\

Our first result is an upper bound on the set of values $\gamma$ so that the process dies out, i.e., reaches the all $0$ state with probability $1$.  The bound depends only on the maximum degree $M = \max_x \deg x$ of the graph.  
\begin{theorem}\label{thmcritmat}
If $\gamma < 1/(2M-1)$ then starting from any finite number of occupied sites, the process dies out, no matter the value of $\lambda$ and $\delta$.
\end{theorem}
This answers question 6 in \cite{krone}, where the author supplies a bound for $\mathbb{Z}^1$ in the case of nearest neighbour interactions, and asks whether a bound exists for other interactions, or for $\mathbb{Z}^d$ with $d>1$.\\

Our next result shows that two notions of survival for the two-stage contact process coincide, answering question 1 in \cite{krone} affirmatively.  For terminology see Sections \ref{critvals} and \ref{eqcrit}.
\begin{theorem}\label{thmcrit}
For the two-stage contact process on $\mathbb{Z}^d$, single-site survival occurs if and only if the upper invariant measure is non-trivial.
\end{theorem}
The proof uses the construction of \cite{crit} to show that for both the process and its dual, single-site survival implies the upper invariant measure is non-trivial.\\

An important question for growth models is that of \emph{complete convergence}, which we show is true for the two-stage contact process, answering question 3 in \cite{krone}.  Here $\lambda_c$ is the critical value for single-site survival as defined in Section \ref{critvals} and $\xi_t$ denotes the process.  The $\Rightarrow$ denotes weak convergence.
\begin{theorem}\label{thmcomp}
If $\lambda > \lambda_c$ then complete convergence holds, i.e.,
\begin{equation*}
\xi_t \Rightarrow \alpha\delta_0 + (1-\alpha)\nu
\end{equation*}
where $\nu$ is the upper invariant measure, $\delta_0$ concentrates on the configuration with all $0$'s and $\alpha = \mathbb{P}(\xi_t \textrm{ dies out })$.
\end{theorem}

We now summarize the organization of the paper.  In Section \ref{secbuild} we construct the process and prove the duality relation.  In Section \ref{critvals} we recall the critical values defined in \cite{krone}.  In Section \ref{corrected} we fill in some missing details in the proof of Proposition 3.6 in \cite{krone}.  In Section \ref{critmat} we prove Theorem \ref{thmcritmat}.  In Section \ref{critmat} we give a sufficient condition for the edge speed of the process in one dimension to characterize survival, providing a partial answer to question 2 in \cite{krone}.  In Section \ref{eqcrit} we prove Theorem \ref{thmcrit}, and in Section \ref{compconv} we prove Theorem \ref{thmcomp}.  We discuss the survival region in Section \ref{surv}, and using the construction from the proof of Theorem \ref{thmcrit} we find that the process dies out on the boundary of the survival region, providing a partial answer to question 5.  We argue that question 4 appears not to have an affirmative answer, and we give some informal arguments as to why this should be so.\\

\section{Construction and duality}\label{secbuild}
We recall briefly the construction of the process.  Here the process $\xi_t$ lives on the state space $\{0,1,2\}^{V}$ where $V$ is the vertex set of an undirected graph $(V,E)$, with $V=\mathbb{Z}^d$ and $E = \{xy:0<\|x-y\|_{\infty}<r\}$ for some $r\geq 1$ being common choices.  The state space is equipped with the partial order $\xi\leq\xi' \Leftrightarrow \xi(x)\leq \xi'(x)$ for each $x \in V$, where $0<1<2$ is the order on the state at each site.  The process is \emph{attractive} if there exists a coupling so that $\xi_0\leq\xi_0' \Rightarrow \xi_t\leq\xi_t'$ for $t>0$.  It is \emph{additive} if $\xi_0 = \xi_0'\vee\xi_0'' \Rightarrow \xi_t = \xi_t'\vee\xi_t''$, where $(\xi\vee\xi')(x) = \max(\xi(x),\xi'(x))$ for each $x$.  It is \emph{monotone increasing} (\emph{decreasing}) with respect to a parameter $\lambda$ if $\xi_0\leq\xi_0'$ and $\lambda\leq\lambda'$ ($\lambda\geq\lambda'$) $\Rightarrow \xi_t\leq\xi_t'$.  We shall often use the word \emph{active} to refer to a site or a point in spacetime where the state is not $0$.\\

We can construct the process on any undirected graph $(V,E)$ by taking $n_2(x)$ to be the cardinality of the set $\{y\in V:xy\in E\}$.  Assign independent Poisson processes to each of the events:
\begin{itemize}
\item death of $1$'s and $2$'s at each site, at rate $1$
\item death of $1$'s at each site at the additional rate $\delta$
\item transmission across each edge at rate $\lambda$
\item maturation at each site, at rate $\gamma$
\end{itemize}
Place the events on the spacetime graph $V\times\mathbb{R}^+$ and fix a configuration at time $0$.  The configuration at later times can then be determined from the events on the graph.  To ensure it is well-defined it suffices to work backwards from a point $(x,t)$ on the spacetime graph and ensure that with probability 1, only finitely many events occur that can influence the state of $(x,t)$.  For this to be true it suffices that the graph has finite maximum degree, i.e., for some $M$ we have $\deg x \leq M<\infty$ for each $x \in V$; the desired property then follows by comparison with a branching process in which births occur at rate $\lambda M$.\\

Additivity of the process is immediate from this construction and from the fact that each transition is additive.  Monotonicity with respect to parameters can be established in the usual way; for example, to compare processes with identical values of $\gamma$ and $\delta$ and  transmission rates $\lambda<\lambda'$ on the same graph, simply add a point process at rate $\lambda'-\lambda$ for the extra transmission events in the second process, and note that this tends to give larger configurations as the process evolves.\\

For each $\delta$, there is a dual process which is given by the rates
\begin{eqnarray*}
0 & \rightarrow & 1 \textrm{ at rate } \lambda n_2\\
1 & \rightarrow & 2 \textrm{ at rate } \gamma\\
2 & \rightarrow & 1 \textrm{ at rate } \delta\\
1,2 & \rightarrow & 0 \textrm{ at rate } 1\\
\end{eqnarray*}
and which Krone calls the ``on-off'' process because of the $2\rightarrow 1$ transition.  Note the dual is similar to the original process, in that type $0$ represents a vacant state, and type $2$'s give birth to type $1$'s.  Define the compatibility relation $\xi\sim\zeta \Leftrightarrow \xi(x)\sim\zeta(x)$ for some $x$, where $1,2\sim 2$ and $2 \sim 1$.  Notice that type $2$ in the dual process corresponds to type $1$ or type $2$ in the original process and that dual type $1$ corresponds to original type $2$.\\

The interpretation of compatibility is that the configuration $\xi$ is strong enough to be compatible with $\zeta$ at some site, and the stronger the dual type, the easier it is to match up.  We give a simple proof of the following fact, the proof of which occupies several pages in \cite{krone}.
\begin{proposition}
\label{dual}
The dual process has the property that
\begin{equation*}
\xi_t \sim \zeta_0 \Leftrightarrow \zeta_t \sim \xi_0
\end{equation*}
with the dual running down the (same) spacetime graph from time $t$ to time $0$, so that $\zeta_s$ is on the time line $t-s$.
\end{proposition}
\begin{proof}
We start from the above condition to construct the dual, showing that it has the stated transitions and rates.  The proof is given for the case $|V|<\infty$, that is, when the set of sites is finite, since only finitely many events occur in a finite time and we can proceed by induction on the events.  To extend this to the case $|V|=\infty$ fix a finite subset $V_0$ and let $V_k = \{y \in V:\dist(y,V_0)\leq k\}$, where $\dist$ is the graph distance.  Denoting by $_{V_k}\xi_t$ the process constructed using the events on the subset $V_k\times\mathbb{R}^+$ of the spacetime graph, there is an almost surely finite value of $k_0$ so that $_{V_k}\xi_s(x) = \xi_s(x)$ for $x \in V_0$ and $0\leq s \leq t$ when $k\geq k_0$, and this suffices to make the extension.\\

Say that a set of (forward) configurations $\Lambda$ is \emph{dualizable} if there is a dual configuration $\zeta$ so that
\begin{equation*}
\Lambda = \{\xi: \xi\sim\zeta\}
\end{equation*}
Note that $\zeta$ is unique, if it exists.  For fixed $\zeta_0$ and $0\leq s \leq t$ let $\Lambda_s = \{\xi_{t-s}:\xi_t\sim\zeta_0\}$.  Clearly, $\Lambda_0$ is dualizable with dual configuration $\zeta_0$.  If $\Lambda_s$ is dualizable, denote by $\zeta_s$ its dual configuration.  Suppose there is an event at time $s$, and use the notation $\xi_{t-s^+}$ and $\Lambda_{s^+}$ to denote the state just prior to its occurrence.  Note that 
\begin{equation*}
\Lambda_{s^+} = \{\xi_{t-s^+}:\xi_{t-s}\in\Lambda_s\}
\end{equation*}
Suppose that $\Lambda_s$ is dualizable with dual configuration $\zeta_s$, then $\Lambda_{s^+} = \{\xi_{t-s^+}:\xi_{t-s}\sim\zeta_s\}$.  We show that $\Lambda_{s^+}$ is dualizable by producing its dual configuration $\zeta_{s^+}$.  A type $2$ death at $x$ (i.e., a rate $1$ death event) kills both active types, so $\zeta_{s^+}(x)=0$ whatever the value of $\zeta_s(x)$; this causes the dual $1,2\rightarrow 0$ transition at rate $1$.  A type $1$ death at $x$ (i.e., a rate $\delta$ death event) kills only type $1$.  If $\zeta_s(x)=2$ i.e., a $1$ or a $2$ is sufficient for compatibility after the event, then a $2$ is required for compatibility before, so $\zeta_{s^+}(x)=1$; this is the dual $2\rightarrow 1$ transition at rate $\delta$.  A (forward) transmission event from $y \rightarrow x$ leads to a $1$ at $x$ after the event, if $y$ is in state 2 just before the event, so $\zeta_{s^+}(y) = 1$ if $\zeta_s(x) = 2$; this is the dual transmission event.  A maturation event at $x$ causes a $1\rightarrow 2$ transition, so that $\zeta_{s^+}(x)=2$ if $\zeta_s(x)=1$; this is the dual $1\rightarrow 2$ transition at rate $\gamma$.  For values of $\zeta_s(x)$ not mentioned, or for sites that aren't involved in the transition, it is easily verified that $\zeta_{s^+}(x)=\zeta_s(x)$.  This finishes the induction step and establishes the dual transitions, completing the proof.
\end{proof}
Before moving on, we note that the dual process is also additive and monotone increasing in $\lambda$ and $\gamma$, and monotone decreasing in $\delta$, a fact which is noted in \cite{krone} and which we use later.

\section{Main Results}\label{secmain}

\subsection{Critical values for survival}\label{critvals}
Denoting by $\xi_t^o$ the process starting a single mature site (the ``$o$'' stands for ``origin'', which if the process lives on the lattice, we can without loss of generality set to be the initially occupied site), we say $\xi_t^o$ survives if
\begin{equation*}
\mathbb{P}(\forall t>0, \exists x: \xi_t^o(x) \neq 0)>0
\end{equation*}
and dies out otherwise.  Defining the critical value
\begin{eqnarray*}
\lambda_c(\gamma,\delta) = \inf\{\lambda>0: \xi_t^o \textrm{ survives }\}
\end{eqnarray*}
it follows by monotonicity that $\lambda_c$ is an increasing function of $\delta$ and a decreasing function of $\gamma$ and that $\xi_t^o$ survives if $\lambda>\lambda_c$.  For each $\delta$, by taking $\gamma$ and $\lambda$ large enough and comparing to a (suitably scaled in time) 1-dependent bond percolation diagram it is possible to show that $\xi_t$ survives, which implies that $\lambda_c(\delta,\gamma)<\infty$ if $\gamma$ is large enough.  The first proof of this type is given by Harris for the contact process in \cite{harris}; its application to the two-stage process is noted in \cite{krone}.\\

For fixed $\delta$ the parameter space for the process is the quadrant $\{(\lambda,\gamma):\lambda\geq 0, \gamma \geq 0\}$, and by identifying the survival region $\mathcal{S} = \{(\lambda,\gamma):\xi_t^o \textrm{ survives }\}$ we obtain a phase diagram for survival.  We can define the critical lines
\begin{eqnarray*}
\lambda_*(\delta) = \inf\{\lambda:\xi_t^o \textrm{ survives for some } (\lambda,\gamma)\}\\
\gamma_*(\delta) = \inf\{\gamma:\xi_t^o \textrm{ survives for some } (\lambda,\gamma)\}
\end{eqnarray*}
that bound the survival region below, and on the left.  From monotonicity it follows that $\lambda_*(\delta)\geq\lambda_c(\infty)$, the critical value for the contact process, and also that $\gamma_*(\delta)\geq\gamma_*(0)$, the left-hand critical line when $\delta=0$.  We shall have more to say about the survival region in Section \ref{surv}.  First, we complete a proof given in \cite{krone} that characterizes $\lambda_*$ for any value of $\delta$.\\

\subsection{Correction to Proposition 3.6}\label{corrected}
In Krone, Proposition 3.6 it is claimed that $\lambda_*(\delta)=\lambda_c(\infty)$ for the process on $\mathbb{Z}^d$, for any dimension $d$.  However, the proof given covers only the case $d=1$.  This is because the paper to which it refers gives a finite spacetime condition for survival only when $\lambda>\lambda_c^{(1)}$, the critical value for the contact process in one dimension.  Here we use the more general construction of \cite{crit}, plus a perturbation argument, to show that $\lambda_*(\delta) \leq \lambda_c(\infty)$ in any dimension, which combined with the previous inequality implies the desired result.\\

In \cite{crit} it is shown for the contact process that if $\lambda>\lambda_c$ and $\epsilon>0$, we can place a latticework structure over an effectively two-dimensional region in $\mathbb{Z}^d\times \mathbb{R}^+$ and make a $1:1$ correspondence between certain spacetime boxes contained in this structure and the set $\{(x,y) \in \mathbb{Z}^2:y\geq 0, x+y \textrm{ is even}\}$ with the property that when the process starts with a large disc of active sites in the box corresponding to $(x,y)$, then with probability $>1-\epsilon$ it can produce a large disc of active sites in the boxes corresponding to both $(x-1,y+1)$ and $(x+1,y+1)$.  In their paper, they then show that if one decreases $\lambda$ slightly, this property still holds, and using results for oriented percolation in two dimensions, conclude that the process still survives.\\

In our case it suffices to show that the property still holds when $\gamma$ is decreased slightly from $\infty$, i.e., when $\gamma$ is large enough.  From this we may then conclude that if $\lambda>\lambda_c(\infty)$ then $\lambda>\lambda_c(\gamma)$ for some $\gamma$, which implies that $\lambda>\lambda_*$, or $\lambda_c(\infty)\geq\lambda_*$, and combining the inequalities, $\lambda_c(\infty)=\lambda_*$.\\

It is sufficient to show that on a finite spacetime region, when $\gamma$ is large enough and the two processes are started from the same configuration (with mature sites in the place of active sites in the two-stage process), with high probability,
\begin{itemize}
\item between any two transmission events incident at a given site, there is a maturation event, and
\item if at a fixed time the contact process has a certain set of active sites, then in the two-stage process those sites are all mature sites
\end{itemize}
The first condition ensures that no connections are cut due to a juvenile site being unable to give birth at a neighbouring site.  The second condition ensures that if the contact process has produced a large disc of active sites, then the two-stage process has produced a large disc of mature sites.\\

To satisfy both conditions, it suffices to ensure that maturation events occur arbitrarily often, since on a finite spacetime region $B \subset \mathbb{Z}^d\times \mathbb{R}^+$, for each $\epsilon>0$ there is a $\delta>0$ so that with probability $>1-\epsilon$, the waiting time between transmission events is $\geq\delta$ everywhere on $B$.  However, for each $\delta>0$ and $\epsilon>0$ there is a $\gamma_0$ so that if $\gamma>\gamma_0$, with probability $>1-\epsilon$ the waiting time between maturation events is $<\delta$ everywhere on $B$, thus for $\gamma>\gamma_0$ the conditions hold.\\

The two assertions of the last paragraph (those regarding waiting times) require proof, and it suffices to consider a spacetime region which is a single interval of length $L$.  To prove the first assertion, notice that with high probability a finite number $N$ of events occur in the interval, and with probability $e^{-\delta\lambda N}$ which $\rightarrow 1$ as $\delta\rightarrow 0$, each event takes time $\geq\delta$ to occur.  To prove the second assertion, break up the interval into pieces of length $\delta$, so that the number of events on each piece is distributed like a Poisson random variable with mean $\delta$.  The probability that on each interval at least one event has occurred is $(1-e^{-\delta\gamma})^{L/\delta}$ which $\rightarrow 1$ as $\gamma\rightarrow\infty$, for fixed $\delta$.

\subsection{Critical maturation rate (q.6)}\label{critmat}
In \cite{krone}, a lower bound on $\gamma_*(\delta)$ is given for the process on $\mathbb{Z}$ with nearest-neighbour interactions which is about $1/4$ when $\delta=0$ and increases towards $1$ as $\delta\rightarrow\infty$.  Here we answer question 6 in that paper, which asks for lower bounds on $\gamma_*(\delta)$ in other settings; we obtain here a simple lower bound on $\gamma_*(0)$ (and by monotonicity, on $\gamma_*(\delta)$) that works for any graph of finite maximum degree, and depends only on the maximum degree.  A graph has finite maximum degree if there is a number $M$ so that $\deg x \leq M < \infty$ for each $x \in V$.
\begin{proposition}\label{gammacrit}
If $\delta=0$ and $\gamma < 1/(2M-1)$ the process dies out for any value of $\lambda$.
\end{proposition}
\begin{proof}
It suffices to show this for $\lambda=\infty$, i.e., when the $0\rightarrow 1$ transition at $x$ is instantaneous if $n_2(x)>0$.  The result is obtained by estimating the average number of offspring of a site $x$ in state 1.  The transition $1\rightarrow 2$ occurs with probability $\gamma/(1+\gamma)$, since $1\rightarrow 0$ at rate $1$ and $1\rightarrow 2$ at rate $\gamma$.  If the $1\rightarrow 2$ transition occurs at $x$, then each unoccupied neighbour of $x$ becomes occupied.  In order for $x$ to send a second offspring to a neighbour $y$, the existing offspring at $y$ has to die off.  Denoting by $N_t$ a Poisson process with rate $1$ (representing the number of deaths of at $y$, starting from the moment the $1\rightarrow 2$ transition occurs at $x$) and by $X_t$ an independent exponential random variable (representing death of the mature organism at $x$), the number of additional offspring produced at $y$ is equal to $N_{X_t}$.  Intuitively, we might expect $\mathbb{E}N_{X_t} = \mathbb{E}N_{\mathbb{E}X_t} = 1$, and computing, we confirm that
\begin{eqnarray*}
\mathbb{E}N_{X_t} &=& \int_0^{\infty}\sum_k k x^k\frac{e^{-x}}{k!}e^{-x}dx\\
&=& \int_0^{\infty}\sum_k k \frac{x^k}{k!}e^{-2x}dx\\
&=& \sum_k k \int_0^{\infty} \frac{x^k}{k!}e^{-2x}dx\\
&=& \sum_k k 2^{-(k+1)}\\
&=& 1
\end{eqnarray*}
Thus the expected number of offspring at each initially unoccupied neighbouring site is $1+\mathbb{E}N_{X_t} = 1 + 1 = 2$, so the expected number of offspring of a site in state 2 is at most $2M$.  Since the probability of making the $1\rightarrow 2$ transition before dying is $\gamma/(1+\gamma) = 1/(1+1/\gamma)$ the expected number of offspring of a site in state 1 is at most $2M/(1+1/\gamma)$.  Setting this $<1$ and comparing to a branching process gives the result.
\end{proof}  

From Proposition \ref{gammacrit} we conclude that $\gamma_*(0) \geq 1/(2M-1)$, so that $\gamma_*(\delta)\geq\gamma_*(0)\geq 1/(2M-1)$ for any $\delta$, proving Theorem \ref{thmcritmat}.  For the nearest-neighbour process on $\mathbb{Z}^d$ we have $M=2d$, giving $\gamma_*(\delta)\geq 1/(4d-1)$, which is $1/3$ for $d=1$, $1/7$ for $d=2$, etc.\\

\subsection{Single-site survival and edge speed (q.2)}
Let $\xi_t^-$ denote the process starting from type $2$'s on $\mathbb{Z}^-$, and let $r_t = \sup\{x:\xi_t^-(x)\neq 0\}$ denote the right edge of $\xi_t^-$.  A result of Durrett shows that $r_t/t$ converges to a constant $\alpha$ as $t\rightarrow\infty$.  It is asked in \cite{krone} (question 2 in Section 4) whether $\lambda_c = \inf\{\lambda:\alpha(\gamma,\delta)>0\}$.  Here we do not prove this, but we give a sufficient condition for it to be true.  To get a sense of what it means, note that this property is equivalent to the property that $\xi_t$ is supercritical (i.e., $\lambda>\lambda_c$) if and only if the right edge of the process started from a half-line of mature sites has a positive spreading speed.  For the equivalence of these statements, note that $\alpha$ is upper semi-continuous in $\lambda$, since it is the infimum of a family of continuous functions as described in \cite{speed}.  \\

One side of the result is easy; letting $\xi_t^+$ denote the process starting from type $2$'s on $\mathbb{Z}^+$ and $\ell_t$ its left edge, by attractiveness $\xi_t^o \leq \min \xi_t^+,\xi_t^-$, so $\xi_t^o(x)=0$ for $x>r_t$ and for $x<l_t$.  If $\alpha<0$ then by symmetry $\ell_t/t \rightarrow -\alpha>0$.  Since $r_t\rightarrow -\infty$ and $l_t\rightarrow \infty$, eventually $r_t<l_t$ and $\xi_t^o(x)=0$ for all $x$, i.e., $\xi_t^o$ dies out.\\

For the converse, for $x\in\mathbb{Z}$ denote by $C_x$ the ``active cluster'' of $x$, i.e., the set of spacetime points $(y,t)$ such that if site $x$ is initially in state $2$, then site $y$ is active at time $t$, and denote by $|C_x|$ its width, that is, $|C_x| = \sup\{|y-x|:(y,t) \in C_x \textrm{ for some }y,t\}$.  If the one-site process $\xi_t^o$ survives then $\mathbb{E}|C_x|=\infty$ for each $x$, thus if $\mathbb{E}|C_x|<\infty$ then $\lambda \leq\lambda_c$.  By analogy with percolation theory \cite{perc} we might guess that the converse holds, i.e., that if $\lambda<\lambda_c$ then $\mathbb{E}|C_x|<\infty$; this is proved, for example, for the contact process in \cite{subcrit}.  We do not pursue this here, but instead show that if $\mathbb{E}|C_x|<\infty$ then $\alpha\leq 0$.  Thus a sufficient condition for edge speed to characterize single-site survival is for the subcritical process to have a finite expected size.\\

\begin{proposition}\label{cluster}
If $\mathbb{E}|C_x|<\infty$ then $\alpha \leq 0$.
\end{proposition}
\begin{proof}
If $\mathbb{E}|C_x|<\infty$ but $\alpha>0$ then each $C_x$ is bounded almost surely, but for each $\epsilon>0$ eventually $r_t/t > \alpha-\epsilon$, which means with probability 1 there is an infinite sequence of sites $(x_k)$ in $\mathbb{Z}^-$ and $(y_k) $ in $\mathbb{Z}^+$ with $x_{k+1}<x_k$ for each $k$, and an infinite sequence of times $(t_k)$ with $t_k\rightarrow\infty$ such that for each $k$, $(x_k,0)\rightarrow (y_k,t_k)$.  This is because the cluster of any finite collection of sites is almost surely bounded, which means that later activity of the process must originate from sites which are progressively further to the left; note that  this implies also that the stated paths must be disjoint, although we will not need this here.  In any case, the event $|C_x| \geq |x|$ occurs for infinitely many $x \in \mathbb{Z}^-$.  However,
\begin{equation*}
\sum_{x \in \mathbb{Z}^-}\mathbb{P}(|C_x| \geq |x| = \mathbb{P}(|C_0 \geq 0|) + \mathbb{E}|C_x| < \infty
\end{equation*}
so applying the Borel-Cantelli lemma, $|C_x| \geq |x|$ occurs infinitely often with probability zero, contradicting our assumption.
\end{proof}

\subsection{Equality of critical values (q.1)}\label{eqcrit}
For any attractive growth model there is at least one other characterization of survival aside from single-site survival, or divergence of the expected cluster size, which is the existence of a non-trivial upper invariant measure $\nu$, obtained as the weak limit of the distribution of the process started from its largest initial configuration (in the case of $\xi_t$, when started from all sites in state 2).  For either the two-stage contact process or the on-off process, this weak limit exists by attractiveness, and from the Feller property is an invariant measure for the system; see \cite{ips}, Chapter, Theorem 2.3 on page 135 for a proof.  The proof is for spin systems but generalizes without modification to any attractive system with a largest configuration.\\

It is possible that $\nu = \delta_0$, the measure that concentrates on the configuration with all 0's; we say that $\nu$ is non-trivial if $\nu \neq \delta_0$, equivalently, if $\nu$ assigns positive density at each site, that is, $\nu\{\xi:\xi(x)\neq 0\} >0$ for each $x$.  In \cite{krone} (question 1 in Section 4), it is asked whether single-site survival is equivalent to this property.  First we show that single-site survival of either the two-stage contact process, or of the on-off process, implies that $\nu \neq \delta_0$, which supplies one direction of the proof.  We then use the duality relation to observe that
\begin{equation*}
\nu(\{\xi:\xi(o)\neq 0\}) = \mathbb{P}(\zeta_t^o \textrm{ survives })
\end{equation*}
where $\xi_t$ is the two-stage contact process and $\zeta_t$ is the on-off process, and that the same property holds when $\xi_t$ and $\zeta_t$ are exchanged in the formula.  Thus if the two-stage contact process has a non-trivial stationary distribution, then the on-off process has single-site survival, which means that the on-off process has a non-trivial stationary distribution, which means that the two-stage contact process has single-site survival, which supplies the other direction of the proof, and shows that the two notions of survival are in fact equivalent, proving Theorem \ref{thmcrit}.  Therefore, it suffices to show that single-site survival of the two-stage process, or of the on-off process, implies that $\nu \neq \delta_0$.\\

For the (single-stage) contact process $\eta_t$ on $\mathbb{Z}^d$, if $\lambda>\lambda_c$ then the method described in \cite{crit} allows us to conclude that under a suitable renormalization and started from a finite number of active sites, $\eta_t$ dominates a supercritical 1-dependent oriented site percolation process in two dimensions, for which it is known that the origin is active for a positive fraction of the time, and from which it follows that $\liminf_t \mathbb{P}(\eta_t^o(0) \neq 0)>0$, which since $\eta_0^o \leq \eta_0^1$ and by attractiveness implies that $\nu(\{\eta:\eta(0) \neq 0\} = \lim_{t\rightarrow\infty}\mathbb{P}(\eta_t^1(0) \neq 0)>0$, where $\eta_t^1$ is the contact process started from all sites active.  The following lemma allows to conclude the same fact for the two-stage contact process, whenever the interaction neighbourhood is symmetric about permutation and sign change of coordinates, and such that with some probability, any site can infect any other site; the first condition we call coordinate symmetry, and the second we call irreducibility.  Note the interaction neighbourhood must of course be finite.

\begin{lemma}\label{crit}
The construction in \cite{crit} is valid for the two-stage contact process and for the on-off process on $\mathbb{Z}^d$, for any coordinate-symmetric and irreducible interaction neighbourhood.
\end{lemma}
\begin{proof}
By following each step of the proof, the construction is seen to apply to these processes; we address the main concerns, but omit the details.  The only real modification is to allow for larger neighbourhoods, and it is already noted in \cite{add} that this modification is not hard.\\

In the construction in \cite{crit}, nearest-neighbour interactions are assumed.  This condition can be relaxed by redefining the ``sides'' of the box to be a region whose width is equal to the interaction range of the process.  In this way, we can control transmission from the sides of the rectangle to the outside world as is done in the nearest-neighbour case.\\

When widening the sides, it is necessary to make sure that a large finite disc can be produced at an extra distance corresponding to the range of the interaction, but this can be prescribed.  We can choose which type to require for the discs; we choose arbitrarily that it consist of type $2$ sites.  Irreducibility is required to ensure that, starting from a single infectious site, all sites in a large finite disc can be made infectious with a certain probability.\\

Coordinate symmetry is implicit in the construction (rectangles are used rather than parallelograms, and all side lengths but the one along the time axis are the same), and it is assumed when proving that the process reaches each orthant of the top and sides of the box with high probability.\\

The construction uses the property that the process dies out if its population dips below a certain value infinitely often.  This is a property that holds for any growth model in which there is a finite number of active types, and each active type dies before reproducing with a certain probability.  The analogous survival argument for the sides of the box follows also from this property (see \cite{sis}, Part 1, Proposition 2.8 for a proof in which the extension is more obvious).\\

For the (usual) contact process, the application of the FKG inequality given in the construction follows from the invariance of positive correlations (see \cite{ips}, Theorem 2.14 on page 80), for which it is sufficient that the process be attractive and that its transitions occur only between comparable states, a property which is true of the two-stage contact process and of the on-off process.  The remainder of the arguments use the Markov or strong Markov properties, are geometrical, or pertain to oriented percolation, and no modification is required.
\end{proof}

It follows from the Lemma and from the discussion preceding it that for the two-stage contact process or the on-off process on $\mathbb{Z}^d$, whenever there is single-site survival ($\lambda>\lambda_c$) the upper invariant measure assigns a positive density at each site, i.e., $\nu(\{\xi:\xi(x)\neq 0\})>0$ for each $x$.  The proof is now complete.\\

\subsection{Complete convergence (q.3)}\label{compconv}
Let $\delta_0$ be the measure that concentrates on the configuration with all 0's, and let $\nu$ be the upper invariant measure introduced in the previous section.  For an attractive growth model, complete convergence means that
\begin{equation*}
\xi_t \Rightarrow \alpha\delta_0 + (1-\alpha)\nu
\end{equation*}
as $t\rightarrow\infty$, where $\Rightarrow$ denotes weak convergence and $\alpha = \mathbb{P}(\xi_t \textrm{ dies out })$.  In \cite{krone} it is asked whether complete convergence holds for the two-stage contact process, when $\lambda>\lambda_c$.  We follow \cite{growth}, Section 5, where the argument is used for the contact process; the idea is originally due to Griffeath \cite{limgriff}.  Fix an arbitrary configuration $\xi_0$, and a dual configuration $\zeta_0$ with finitely many active sites; doing this for all such $\zeta_0$ we will recover the finite-dimensional distributions of the upper invariant measure.  We have that
\begin{equation*}
\xi_{2t} \sim \zeta_0 \Leftrightarrow \xi_t \sim \zeta_t
\end{equation*}
where $\zeta_s$, $0\leq s \leq t$ is constructed on the same spacetime graph as $\xi_t$ and run from time $2t$ down to time $t$, with initial configuration $\zeta_0$.  Use the notation $\xi \neq 0$ to denote ``not identically zero''.  Then observe that
\begin{equation*}
\mathbb{P}(\xi_t \sim \zeta_t) = \mathbb{P}(\xi_t \neq 0, \zeta_t \neq 0) - \mathbb{P}(\xi_t \neq 0,\zeta_t \neq 0, \xi_t \nsim \zeta_t)
\end{equation*}
Since they are built over disjoint parts of the graph, $\xi_s$, $0\leq s \leq t$ and $\zeta_s$, $0\leq s \leq t$ are independent, so
\begin{equation*}
\mathbb{P}(\xi_t \neq 0, \zeta_t \neq 0) = \mathbb{P}(\xi_t \neq 0)\mathbb{P}(\zeta_t \neq 0)
\end{equation*}
for each $t>0$.  Using the duality relation, $\mathbb{P}(\zeta_t \neq 0) = \mathbb{P}(\xi_t \sim \zeta_0)$ with $\xi_0$ in this case being the configuration with all $2$'s.  Letting $t\rightarrow\infty$
\begin{equation*}
\mathbb{P}(\xi_t \neq 0)\mathbb{P}(\zeta_t \neq 0) \rightarrow (1-\alpha)\nu(\{\xi:\xi\sim\zeta_0\})
\end{equation*}
To have complete convergence, it therefore suffices to show that
\begin{equation*}
\mathbb{P}(\xi_t \neq 0,\zeta_t \neq 0,\xi_t \nsim \zeta_t) \rightarrow 0
\end{equation*}
as $t\rightarrow\infty$.  A method for doing this is outlined in \cite{growth} for a certain class of growth models.  They use a restart argument to show that whenever the process survives, and suitably rescaled, it eventually dominates a two-dimensional oriented percolation process, which is known to have a positive density of sites.  Using this fact it is then argued that if run for long enough, the process and its dual intersect with high probability.  In Section 5 of \cite{crit} it is noted that, using their construction and the ideas from \cite{growth}, the same can be concluded for the contact process in $\mathbb{Z}^d$.  Noting the equality of critical values proved in the previous section, if $\lambda>\lambda_c$ the construction of \cite{crit} can be applied to both the two-stage contact process and the on-off process, and the same arguments apply to show that the two processes eventually intersect with high probability, which proves Theorem \ref{thmcomp}.\\

\subsection{Structure of the survival region (q.5 and q.4)}\label{surv}
Continuing the analysis of the survival region $\mathcal{S}$ begun in Section \ref{critvals}, we show that the process dies out on the boundary $\partial S$.  By monotonicity $\lambda_c(\gamma)$ can have only jump discontinuities, which means that the boundary of the survival region is the set
\begin{equation*}
\{(\gamma,\lambda):\gamma\geq\gamma_*,\lambda_c^-(\gamma)\geq\lambda\geq\lambda_c^+(\gamma)\}
\end{equation*}
where $\lambda_c^-(\gamma)$ and $\lambda_c^+(\gamma)$ are the left-hand and right-hand limits of $\lambda_c$ at $\gamma$; set $\lambda_c^-(\gamma_*)=\infty$.\\

It follows from Lemma \ref{crit} that survival of the two-stage contact process is given by a finite spacetime condition of the form ``a certain event happens with probability $>1-\epsilon$'', where $\epsilon$ is sufficiently small.  Moreover, the probability of this event is continuous in $\lambda$ and $\gamma$ (also $\delta$, but we will not use this fact here).  This is because by a small enough change in parameters, on a finite spacetime region we can ensure that the probability of even one more or one fewer transmission/maturation events can be made arbitrarily small.  This implies that $\mathcal{S}$ is an open subset of the plane, in any dimension and for any value of $\delta$.  Since $\mathcal{S}$ is open it follows that the process dies out on its boundary $\partial S$, which includes the critical values $\lambda_c(\gamma)$.\\

It seems that question 4, namely, whether there is a formula for $\lambda_c$ in terms of $\lambda_c(\infty)$, $\gamma$ and $\delta$, should be false.  One good reason to believe this is that for the contact process, a sequence of approximants is known that converges to the critical value, and these are roots of successively more complicated rational functions, as shown in \cite{ziezold}.  There is no obvious reason to believe why the introduction of an additional stage to the process should lead to a critical value which is any simpler to determine, even if the critical value of the contact process is used in the expression.\\

Remaining questions for the survival region include whether $\lambda_c(\gamma)$ is continuous, whether it is strictly decreasing on $\{\gamma > \gamma_*\}$ and whether
\begin{equation*}
\lim_{\gamma\rightarrow \gamma_*^+}\lambda_c(\infty) = \infty
\end{equation*}
We believe the answers are respectively yes, yes, and yes, but we are not sure how to prove this.

\bibliography{twoStage}
\bibliographystyle{plain}
\end{document}